\documentclass[12pt,reqno]{amsart}
\usepackage{amsmath,amssymb}
\setbox0=\hbox{$+$}
\newdimen\plusheight
\plusheight=\ht0
\def\+{\;\lower\plusheight\hbox{$+$}\;}

\setbox0=\hbox{$-$}
\newdimen\minusheight
\minusheight=\ht0
\def\-{\;\lower\minusheight\hbox{$-$}\;}

\setbox0=\hbox{$\cdots$}
\newdimen\cdotsheight
\cdotsheight=\plusheight
\def\cds{\lower\cdotsheight\hbox{$\cdots$}}

\renewcommand{\(}{\left\(}
\renewcommand{\)}{\right\)}
\renewcommand{\[}{\left[}

\theoremstyle{plain}
\newtheorem{thm}{Theorem}[section]
\newtheorem{cor}[thm]{Corollary}
\newtheorem{lem}[thm]{Lemma}

\theoremstyle{definition}
\newtheorem{defin}[thm]{Definition}

\begin{document}
\title[$p$-adic Gamma function and the trace of Frobenius of elliptic curves]
{$p$-adic Gamma function and the trace of Frobenius of elliptic curves}
\author{Rupam Barman}
\address{Department of Mathematics, Indian Institute of Technology, Hauz Khas, New Delhi-110016, INDIA}
\email{rupambarman@gmail.com}
\author{Neelam Saikia}
\address{Department of Mathematical Sciences, Tezpur University, Napaam-784028, Sonitpur, Assam, INDIA}
\email{nlmsaikia1@gmail.com} \vspace*{0.7in}
\begin{center}

{\bf $p$-ADIC GAMMA FUNCTION AND THE TRACE OF FROBENIUS OF ELLIPTIC CURVES}\\[5mm]

Rupam Barman and Neelam Saikia\\[.2cm]
\end{center}
\vspace{.51cm} \noindent\textbf{Abstract}: In \cite{mccarthy2}, McCarthy defined a function $_{n}G_{n}[\cdots]$ using Teichm\"{u}ller
character of finite fields and quotients of $p$-adic gamma function, and expressed the trace of Frobenius of elliptic curves in terms of
special values of $_{2}G_{2}[\cdots]$. We establish two different expressions for the traces of Frobenius of elliptic curves in terms of the function
$_{2}G_{2}[\cdots]$. As a result, we obtain two relations between special values of the function $_{2}G_{2}[\cdots]$ with different parameters.

\vspace{.2cm}
\noindent{\footnotesize \textbf{Key Words}: Trace of Frobenius; elliptic curves; characters of finite fields; Gauss sums; Teichm\"{u}ller character;
$p$-adic Gamma function.}

\vspace{.2cm}
\noindent{\footnotesize 2010 Mathematics Classification Numbers: Primary: 11G20, 33E50; Secondary: 33C99, 11S80,
11T24.}
\section{Introduction and statement of results}
Let $q=p^r$ be a power of an odd prime, and let $\mathbb{F}_q$ be the finite field of $q$ elements. 
Let $\mathbb{Z}_p$ denote the ring of $p$-adic integers. Let $\Gamma_p(.)$ denote the Morita's $p$-adic gamma function, and let $\omega$ denote the 
Teichm\"{u}ller character of $\mathbb{F}_q$. We denote by $\overline{\omega}$ the inverse of $\omega$. 
For $x \in \mathbb{Q}$ we let $\lfloor x\rfloor$ denote the greatest integer less than 
or equal to $x$ and $\langle x\rangle$ denote the fractional part of $x$, i.e. $x-\lfloor x\rfloor$. Also, we denote by $\mathbb{Z}^{+}$ and $\mathbb{Z}_{\geq 0}$
the set of positive integers and non negative integers, respectively. In \cite{mccarthy2}, McCarthy defined a function 
$_{n}G_{n}[\cdots]$ as given below.
\begin{defin}\cite[Defn. 5.1]{mccarthy2} \label{defin1}
Let $q=p^r$, for $p$ an odd prime and $r \in \mathbb{Z}^+$, and let $t \in \mathbb{F}_q$. 
For $n \in \mathbb{Z}^+$ and $1\leq i\leq n$, let $a_i$, $b_i$ $\in \mathbb{Q}\cap \mathbb{Z}_p$. Then the function $_{n}G_{n}[\cdots]$ is defined by 
\begin{align}
&_nG_n\left[\begin{array}{cccc}
             a_1, & a_2, & \cdots, & a_n \\
             b_1, & b_2, & \cdots, & b_n
           \end{array}|t
 \right]_q:=\frac{-1}{q-1}\sum_{j=0}^{q-2}(-1)^{jn}~~\overline{\omega}^j(t)\notag\\
&\times \prod_{i=1}^n\prod_{k=0}^{r-1}(-p)^{-\lfloor \langle a_ip^k \rangle-\frac{jp^k}{q-1} \rfloor -\lfloor\langle -b_ip^k \rangle +\frac{jp^k}{q-1}\rfloor}
 \frac{\Gamma_p(\langle (a_i-\frac{j}{q-1})p^k\rangle)}{\Gamma_p(\langle a_ip^k \rangle)}
 \frac{\Gamma_p(\langle (-b_i+\frac{j}{q-1})p^k \rangle)}{\Gamma_p(\langle -b_ip^k \rangle)}.\notag
\end{align}
\end{defin}
This function has many interesting properties. For further details, see \cite{mccarthy2}.
In \cite{greene}, Greene introduced the notion of hypergeometric
functions over finite fields. Since then, many interesting connections between hypergeometric functions over finite field and algebraic curves have been found.
But these results are restricted to primes satisfying certain congruence conditions. For example, see \cite{BK1, BK2, Fuselier, lennon, lennon2}.  
Let $E/\mathbb{F}_q$ be an elliptic curve given in the Weierstrass form.
Then the trace of Frobenius $a_q(E)$ of $E$ is given by
\begin{align}\label{eq18}
  a_q(E):=q+1-\#E(\mathbb{F}_q),
\end{align}
where $\#E(\mathbb{F}_q)$ denotes the number of $\mathbb{F}_q$-points on $E$ including the point at infinity. Let $j(E)$ denote the $j$-invariant of the elliptic 
curve $E$. Let $\phi$ be the quadratic character of $\mathbb{F}_q^{\times}$ extended to all of $\mathbb{F}_q$ by setting $\phi(0):=0$. 
Using the function $_{2}G_{2}[\cdots]$, McCarthy expressed the trace of 
Frobenius of elliptic curves defined over $\mathbb{F}_p$ without any congruence condition on the prime. The statement of his result is given below.
\begin{thm}\cite[Thm. 1.2]{mccarthy2}\label{mc}
Let $p>3$ be a prime. Consider an elliptic curve $E_s/\mathbb{F}_p$ of the form $E_s: y^2=x^3+ax+b$ with $j(E_s)\neq 0, 1728$. Then
\begin{align}
 a_p(E_s)=\phi(b)\cdot p\cdot {_2}G_2\left[ \begin{array}{cc}
              \frac{1}{4}, & \frac{3}{4} \\
              \frac{1}{3}, & \frac{2}{3}
            \end{array}|-\frac{27b^2}{4a^3}
 \right]_p.
\end{align}
\end{thm}
In \cite{BK2}, the first author \& Kalita gave two formulas for the trace of Frobenius of the elliptic curve $E_{a,b}: y^2=x^3+ax+b$ 
defined over $\mathbb{F}_q$ under the conditions $q\equiv 1$ $($mod $6)$ and  $q\equiv 1$ $($mod $4)$, respectively. In this paper, we prove the following two 
expressions for the trace of Frobenius of the elliptic curve $E_{a,b}/\mathbb{F}_q$ in terms of special values of the function $_{2}G_{2}[\cdots]$
without any congruence conditions on $q$.
\begin{thm}\label{thm1}
Let $q=p^r$, $p>3$ be a prime. Consider an elliptic curve $E_{a, b}/\mathbb{F}_q$ of the form $E_{a, b}: y^2=x^3+ax+b$ with $j(E_{a, b})\neq 0$. 
If $(-a/3)$ is a quadratic residue in $\mathbb{F}_q$, then
\begin{align}
a_q(E_{a,b})=\phi(k^3+ak+b)\cdot q \cdot {_2}G_2\left[ \begin{array}{cc}
              \frac{1}{2}, & \frac{1}{2} \\
              \frac{1}{3}, & \frac{2}{3}
            \end{array}|-\frac{k^3+ak+b}{4k^3}
 \right]_q,\notag
\end{align}
where $3k^2+a=0.$
\end{thm}
\begin{thm}\label{thm2}
Let $q=p^r$, $p>3$ be a prime. Consider an elliptic curve $E_{a, b}/\mathbb{F}_q$ of the form $E_{a, b}: y^2=x^3+ax+b$ with $j(E_{a, b})\neq 1728$. 
If $x^3+ax+b=0$ has a non zero solution in $\mathbb{F}_q$, then 
\begin{align}
a_q(E_{a,b})=\phi(-3h^2-a)\cdot q \cdot {_2}G_2\left[ \begin{array}{cc}
              \frac{1}{2}, & \frac{1}{2} \\
              \frac{1}{4}, & \frac{3}{4}
            \end{array}|\frac{4(3h^2+a)}{9h^2}
 \right]_q,\notag
\end{align}
where $h^3+ah+b=0$.
\end{thm}
McCarthy proved the Theorem \ref{mc} over $\mathbb{F}_p$. Along the proof of the Theorem \ref{thm1}
and Theorem \ref{thm2} (which are proved for $\mathbb{F}_q$), we have verified that the Theorem \ref{mc} is also true for $\mathbb{F}_q$. 
Hence, we have the following corollary which gives nice transformation formulas between special values of the function $_{2}G_{2}[\cdots]$ 
with different parameters.
\begin{cor}\label{cor1}
 Let $q=p^r$, $p>3$ be a prime. Let $a, b \in \mathbb{F}_q^{\times}$. Then\\
${_2}G_2\left[ \begin{array}{cc}
              \frac{1}{4}, & \frac{3}{4} \\
              \frac{1}{3}, & \frac{2}{3}
            \end{array}|-\frac{27b^2}{4a^3}
 \right]_q=\left\{
                                  \begin{array}{ll}
                                    \phi(b(k^3+ak+b))\cdot {_2}G_2\left[ \begin{array}{cc}
                                                         \frac{1}{2}, & \frac{1}{2} \\
                                                         \frac{1}{3}, & \frac{2}{3}
                                                       \end{array}|-\frac{k^3+ak+b}{4k^3}\right]_q \hbox{if~ $a=-3k^2$;} \\
                                    \phi(-b(3h^2+a))\cdot {_2}G_2\left[ \begin{array}{cc}
              \frac{1}{2}, & \frac{1}{2} \\
              \frac{1}{4}, & \frac{3}{4}
            \end{array}|\frac{4(3h^2+a)}{9h^2}
 \right]_q  \hbox{if ~$h^3+ah+b=0$.}
                                  \end{array}
                                \right.$
\end{cor}

\section{Preliminaries}
Let $\widehat{\mathbb{F}_q^\times}$ denote the group of multiplicative characters of $\mathbb{F}_q^\times$. We extend the domain of each 
$\chi \in \mathbb{F}_q^{\times}$ to $\mathbb{F}_q$ by setting $\chi(0):=0$ including the trivial character $\varepsilon$. 
The \emph{orthogonality relations} for multiplicative characters are listed in the following lemma.
\begin{lem}\emph{(\cite[Chapter 8]{ireland}).}\label{lemma2} We have
\begin{enumerate}
\item $\displaystyle\sum_{x\in\mathbb{F}_q}\chi(x)=\left\{
                                  \begin{array}{ll}
                                    q-1 & \hbox{if~ $\chi=\varepsilon$;} \\
                                    0 & \hbox{if ~~$\chi\neq\varepsilon$.}
                                  \end{array}
                                \right.$
\item $\displaystyle\sum_{\chi\in \widehat{\mathbb{F}_q^\times}}\chi(x)~~=\left\{
                            \begin{array}{ll}
                              q-1 & \hbox{if~~ $x=1$;} \\
                              0 & \hbox{if ~~$x\neq1$.}
                            \end{array}
                          \right.$
\end{enumerate}
\end{lem}
\par Let $\mathbb{Z}_p$ denote the ring of $p$-adic integers, $\mathbb{Q}_p$ the field of $p$-adic numbers, $\overline{\mathbb{Q}_p}$ 
the algebraic closure of $\mathbb{Q}_p$, and $\mathbb{C}_p$ the completion of $\overline{\mathbb{Q}_p}$. 
Let $\mathbb{Z}_q$ be the ring of integers in the unique unramified extension of $\mathbb{Q}_p$ with residue field $\mathbb{F}_q$. 
Recall that $\mathbb{Z}_q^{\times}$ contains all $(q-1)$-th root of unity. Therefore, we can consider multiplicative charcaters of $\mathbb{F}_q^\times$
to be maps $\chi: \mathbb{F}_q^{\times} \rightarrow \mathbb{Z}_q^{\times}$.
\par We now introduce some properties of Gauss sums. For further details, see \cite{evans}. Let $\zeta_p$ be a fixed primitive root of unity 
in $\overline{\mathbb{Q}_p}$. Then the additive character 
$\theta: \mathbb{F}_q \rightarrow \mathbb{Q}_p(\zeta_p)$ is defined by
\begin{align}
\theta(\alpha)=\zeta_p^{\text{tr}(\alpha)},\notag
\end{align}
where $\text{tr}: \mathbb{F}_q \rightarrow \mathbb{F}_p$ is the trace map given by
$$\text{tr}(\alpha)=\alpha + \alpha^p + \alpha^{p^2}+ \cdots + \alpha^{p^{r-1}}.$$
For $\chi \in \widehat{\mathbb{F}_q^\times}$, the \emph{Gauss sum} is defined by
\begin{align}
G(\chi):=\sum_{x\in \mathbb{F}_q}\chi(x)\theta(x).\notag
\end{align}
We let $T$ denote a fixed generator of $\widehat{\mathbb{F}_q^\times}$. The Gauss sum $G(T^m)$ is denoted by $G_m$. 
The following lemma provides a formula for the multiplicative inverse of a Gauss sum.
\begin{lem}\emph{(\cite[Eqn. 1.12]{greene}).}\label{fusi3}
If $k\in\mathbb{Z}$ and $T^k\neq\varepsilon$, then
$$G_kG_{-k}=qT^k(-1).$$
\end{lem}
Using orthogonality, we can write $\theta$ in terms of Gauss sums as given in the following lemma.
\begin{lem}\emph{(\cite[Lemma 2.2]{Fuselier}).}\label{lemma1}
For all $\alpha \in \mathbb{F}_q^{\times}$, $$\theta(\alpha)=\frac{1}{q-1}\sum_{m=0}^{q-2}G_{-m}T^m(\alpha).$$
\end{lem}
\begin{thm}\emph{(Davenport-Hasse Relation \cite{Lang}).}\label{lemma3}
Let $m$ be a positive integer and let $q=p^r$ be a prime power such that $q\equiv 1 (\text{mod}~m)$. For multiplicative characters 
$\chi, \psi \in \widehat{\mathbb{F}_q^\times}$, we have
\begin{align}
\prod_{\chi^m=1}G(\chi \psi)=-G(\psi^m)\psi(m^{-m})\prod_{\chi^m=1}G(\chi).
\end{align}
\end{thm}
\par
We now recall the definition of $p$-adic gamma function. For $n \in\mathbb{Z}^+$, 
the $p$-adic gamma function $\Gamma_p(n)$ is defined as
\begin{align}
\Gamma_p(n):=(-1)^n\prod_{0<j<n,p\nmid j}j\notag
\end{align}
and one extends it to all $x\in\mathbb{Z}_p$ by setting $\Gamma_p(0):=1$ and
\begin{align}
\Gamma_p(x):=\lim_{n\rightarrow x}\Gamma_p(n)\notag
\end{align}
for $x\neq0$, where $x$ runs through any sequence of positive integers $p$-adically approaching $x$. This limit exists, is independent of how $n$ approaches $x$, 
and determines a continuous function on $\mathbb{Z}_p$ with values in $\mathbb{Z}_p^{\times}$. 
\par We now state a product formula for the $p$-adic gamma function from \cite[Theorem 3.1]{gross}. 
Let $\omega: \mathbb{F}_q^\times \rightarrow \mathbb{Z}_q^{\times}$ be the Teichm\"{u}ller character. 
For $a\in\mathbb{F}_q^\times$, the value $\omega(a)$ is just the $(q-1)$-th root of unity in $\mathbb{Z}_q$ such that $\omega(a)\equiv a$ $($mod $p)$. 
We denote by $\overline{\omega}$ the inverse of $\omega$. 
If $m\in\mathbb{Z}^+$, $p\nmid m$ and $x$ satisfies $0 \leq x \leq 1$ and $(q-1)x\in \mathbb{Z}$, then
\begin{align}\label{eq2}
\prod_{i=0}^{r-1}\prod_{h=0}^{m-1}\Gamma_p\left(\langle(\frac{x+h}{m})p^i\rangle\right)=\omega \left( m^{(1-x)(1-q)}\right)
\prod_{i=0}^{r-1}\Gamma_p(\langle xp^i\rangle)
\prod_{h=1}^{m-1}\Gamma_p\left(\langle\frac{hp^i}{m}\rangle\right).
\end{align}
We also note that
\begin{align}\label{eq26}
\Gamma_p(x)\Gamma_p(1-x)=(-1)^{x_0},
\end{align}
where $x_0\in \{1, 2, \ldots, p \}$ satisfies $x_0\equiv x$ $($mod $p)$. 
\par The Gross-Koblitz formula allow us to relate the Gauss sums and the $p$-adic gamma function. 
Let $\pi \in \mathbb{C}_p$ be the fixed root of $x^{p-1} + p=0$ which satisfies 
$\pi \equiv \zeta_p-1$ $($mod $(\zeta_p-1)^2)$. Then we have the following result.
\begin{thm}\emph{(Gross, Koblitz \cite{gross}).}\label{thm4} For $a\in \mathbb{Z}$ and $q=p^r$,
\begin{align}
G(\overline{\omega}^a)=-\pi^{(p-1)\sum_{i=0}^{r-1}\langle\frac{ap^i}{q-1} \rangle}\prod_{i=0}^{r-1}\Gamma_p\left(\langle \frac{ap^i}{q-1} \rangle\right).\notag
\end{align}
\end{thm}
\section{Proof of the results}
\par We first prove a lemma which we will use to prove the main results.
\begin{lem}\label{lemma4}
Let $p$ be a prime and $q=p^r$. For $0\leq j\leq q-2$ and $t\in \mathbb{Z^+}$ with $p\nmid t$, we have
\begin{align}\label{eq8}
\omega(t^{tj})\prod_{i=0}^{r-1}\Gamma_p\left(\langle \frac{tp^ij}{q-1}\rangle\right)
\prod_{h=1}^{t-1}\Gamma_p\left(\langle\frac{hp^i}{t}\rangle\right)
=\prod_{i=0}^{r-1}\prod_{h=0}^{t-1}\Gamma_p\left(\langle\frac{p^ih}{t}+\frac{p^ij}{q-1}\rangle\right)
\end{align}
and
\begin{align}\label{eq9}
\omega(t^{-tj})\prod_{i=0}^{r-1}\Gamma_p\left(\langle\frac{-tp^ij}{q-1}\rangle\right)
\prod_{h=1}^{t-1}\Gamma_p\left(\langle \frac{hp^i}{t}\rangle\right)
=\prod_{i=0}^{r-1}\prod_{h=0}^{t-1}\Gamma_p\left(\langle\frac{p^i(1+h)}{t}-\frac{p^ij}{q-1}\rangle \right).
\end{align}
\end{lem}
\begin{proof}
Fix $0\leq j\leq q-2$, and let $k\in\mathbb{Z}_{\geq0}$ be defined such that
\begin{align}\label{eq1}
k\left( \frac{q-1}{t} \right)\leq j < (k+1)\left( \frac{q-1}{t} \right).
\end{align}
Putting $m=t$ and $x=\frac{tj}{q-1}-k$ in \eqref{eq2}, we obtain
\begin{align}\label{eq3}
&\prod_{i=0}^{r-1}\prod_{h=0}^{t-1}\Gamma_p\left( \langle(\frac{j}{q-1}+\frac{h-k}{t})p^i\rangle \right)\notag\\
&~~~~~~~~~~~=\omega\left( t^{(1-\frac{tj}{q-1}+k)(1-q)} \right)
\prod_{i=0}^{r-1}\Gamma_p\left(\langle(\frac{tj}{q-1}-k)p^i\rangle \right)\prod_{h=1}^{t-1}\Gamma_p\left(\langle\frac{hp^i}{t}\rangle \right).
\end{align}
We observe that $0\leq k <t$. 
Therefore, we have
\begin{align}\label{eq4}
&\prod_{i=0}^{r-1}\prod_{h=0}^{t-1}\Gamma_p\left(\langle(\frac{h-k}{t}+\frac{j}{q-1})p^i\rangle\right)\notag\\
&=\prod_{i=0}^{r-1}\prod_{h=0}^{k-1}\Gamma_p\left(\langle(\frac{t+h-k}{t}+\frac{j}{q-1})p^i\rangle\right)
\prod_{h=k}^{t-1}\Gamma_p\left(\langle(\frac{h-k}{t}+\frac{j}{q-1})p^i\rangle\right)\notag\\
&=\prod_{i=0}^{r-1}\prod_{h=t-k}^{t-1}\Gamma_p\left(\langle(\frac{h}{t}+\frac{j}{q-1})p^i\rangle\right)
\prod_{h=0}^{t-k-1}\Gamma_p\left(\langle(\frac{h}{t}+\frac{j}{q-1})p^i\rangle\right)\notag\\
&=\prod_{i=0}^{r-1}\prod_{h=0}^{t-1}\Gamma_p\left(\langle(\frac{h}{t}+\frac{j}{q-1})p^i\rangle\right).
\end{align}
 Again by our choice of $k$, for any nonnegative integer $i$ we have
\begin{align}
\langle\frac{p^itj}{q-1}\rangle=\langle(\frac{tj}{q-1}-k)p^i\rangle.\notag
\end{align}
This gives us
\begin{align}\label{eq6}
\Gamma_p\left(\langle(\frac{p^itj}{q-1})\rangle\right)=\Gamma_p\left(\langle(\frac{tj}{q-1}-k)p^i\rangle\right).
\end{align}
Now substituting \eqref{eq4}, \eqref{eq6} into \eqref{eq3} we obtain \eqref{eq8}.\\\\
We prove \eqref{eq9} following \cite[Lemma 4.1]{mccarthy2} and using similar arguments as given in the proof of \eqref{eq8}.
\end{proof}
\begin{lem}\label{lemma5}
For $1\leq l\leq q-2$ and $0\leq i\leq r-1$, we have
\begin{align}
&\lfloor-\frac{lp^i}{q-1}\rfloor -\lfloor-\frac{2lp^i}{q-1}\rfloor-
\lfloor-\frac{2lp^i}{q-1}\rfloor-\lfloor\frac{3lp^i}{q-1}\rfloor-1\notag\\
&=-2\lfloor\langle \frac{p^i}{2}\rangle- \frac{lp^i}{q-1}\rfloor
-\lfloor\langle- \frac{p^i}{3} \rangle+ \frac{lp^i}{q-1}\rfloor-\lfloor\langle
-\frac{2p^i}{3} \rangle+\frac{lp^i}{q-1}\rfloor.\notag
\end{align}
\end{lem}
\begin{proof}
We can prove the lemma by considering the following cases:\\\\
Case 1: $1\leq l<\frac{q-1}{6p^i}$.\\
Case 2: $\lfloor\frac{q-1}{6p^i} \rfloor<l\leq q-2$.\\
In case 2, we observe that $\lfloor \frac{6lp^i}{q-1} \rfloor \in \{1, 2, \ldots, (6p^i-1)\}$. Now taking $x\in\mathbb{Z}^{+}$ such that
$1\leq x\leq 6p^i-1$ and $x=6u+v$, where $v=0, 1, 2, 3, 4$ or $5$, the result follows.
\end{proof}
\begin{lem}\label{lemma6}
For $0\leq l\leq q-2$ and $0\leq i\leq r-1$, we have
\begin{align}
&\lfloor\frac{2lp^i}{q-1} \rfloor +2\lfloor\frac{-lp^i}{q-1}\rfloor
-2\lfloor \frac{-2lp^i}{q-1} \rfloor-\lfloor\frac{4lp^i}{q-1} \rfloor\notag\\
&=-2\lfloor\langle \frac{p^i}{2}\rangle  -\frac{lp^i}{q-1}\rfloor-
\lfloor\langle-\frac{p^i}{4} \rangle+\frac{lp^i}{q-1} \rfloor
-\lfloor\langle-\frac{3p^i}{4} \rangle+\frac{lp^i}{q-1} \rfloor.\notag
\end{align}
\end{lem}
\begin{proof}
We can prove the lemma by considering the following cases:\\\\
Case 1: $0\leq l<\frac{q-1}{4p^i}$.\\
Case 2: $\lfloor\frac{q-1}{4p^i}\rfloor<l\leq q-2$.\\
In case 2, we observe that $\lfloor \frac{4lp^i}{q-1}\rfloor\in \{1, 2, \ldots, (4p^i-1)\}$. Now taking $x\in\mathbb{Z}^{+}$ such that
$1\leq x\leq 4p^i-1$ and $x=4u+v$, where $v=0,1,2$ or $3$, the desired result follows.
\end{proof}
Now, we are going to prove Theorem \ref{thm1}. The proof will follow as a consequence of the next theorem. 
We consider an elliptic curve $E_1$ over $\mathbb{F}_q$ in the form
\begin{align}
E_1 : y^2=x^3+cx^2+d,\notag
\end{align}
where $c\neq0$. 
We express the trace of Frobenius endomorphism on the curve $E_1$ as a special value of the function ${_2}G_2[\cdots]$ in the following way.
\begin{thm}\label{thm3}
Let $q=p^r$, $p>3$ be a prime. The trace of Frobenius on $E_1$ is given by
\begin{align}
a_q(E_1)=q \cdot\phi(d)\cdot {_2}G_2\left[ \begin{array}{cc}
                          \frac{1}{2}, & \frac{1}{2} \\
                          \frac{1}{3}, & \frac{2}{3}
                        \end{array}|-\frac{27d}{4c^3}
 \right]_q.\notag
\end{align}
\end{thm}
\begin{proof}
We have $\#E_1(\mathbb{F}_q)-1=\#\{(x,y)\in\mathbb{F}_q\times\mathbb{F}_q : y^2=x^3+cx^2+d\}.$\\
Let $P(x,y)=x^3+cx^2+d-y^2$. Now using the identity
\begin{align}
\sum_{z\in\mathbb{F}_q}\theta(zP(x,y))=\left\{
                                         \begin{array}{ll}
                                           q, & \hbox{if $P(x,y)=0$;} \\
                                           0, & \hbox{if $P(x,y)\neq0$,}
                                         \end{array}
                                       \right.
\end{align}
we obtain
\begin{align}\label{eq10}
q.(\#E_1(\mathbb{F}_q)-1)&=\sum_{x,y,z\in\mathbb{F}_q}\theta(zP(x,y))\notag\\
&=q^2+\sum_{z\in\mathbb{F}_{q}^{\times}}\theta(zd)+\sum_{y,z\in\mathbb{F}_{q}^{\times}}\theta(zd)\theta(-zy^2)
+\sum_{x,z\in\mathbb{F}_{q}^{\times}}\theta(zd)\theta(zx^3)\theta(zcx^2)\notag\\
&~~~~~~~~~~~~~~~~~~~~~~~~~~~+\sum_{x,y,z\in\mathbb{F}_{q}^{\times}}\theta(zd)\theta(zx^3)\theta(zcx^2)\theta(-zy^2)\notag\\
&=q^2+A+B+C+D.
\end{align}
From the proof of \cite[Theorem 3.1]{BK2}, we have $A=-1$, $B=1+qT^{\frac{q-1}{2}}(d)$ and $D=-C+D_{\frac{q-1}{2}}$ ,
where
\begin{align}
D_{\frac{q-1}{2}}&=\frac{1}{(q-1)^3}\sum_{l,m,n=0}^{q-2}G_{-l}G_{-m}G_{-n}G_{\frac{q-1}{2}}T^l(d)T^n(c)T^{\frac{q-1}{2}}(-1)\notag\\
&\sum_{x\in\mathbb{F}_{q}^{\times}}T^{3m+2n}(x)\sum_{z\in\mathbb{F}_{q}^{\times}}T^{l+m+n+\frac{q-1}{2}}(z),\notag
\end{align}
which is non zero only if $m=-\frac{2}{3}n$ and $n=-3l-\frac{3(q-1)}{2}$. Since $G_{3l+\frac{3(q-1)}{2}}=G_{3l+\frac{q-1}{2}}$
and $G_{-2l-(q-1)}=G_{-2l}$, we have
\begin{align}\label{eq11}
D_{\frac{q-1}{2}}&=\frac{1}{q-1}\sum_{l=0}^{q-2}G_{-l}G_{-2l}G_{3l+\frac{q-1}{2}}G_{\frac{q-1}{2}}T^l(d)T^{-3l+\frac{q-1}{2}}(c)T^{\frac{q-1}{2}}(-1).
\end{align}
Replacing $l$ by $l-\frac{q-1}{2}$ we obtain
\begin{align}\label{eq25}
D_{\frac{q-1}{2}}&=\frac{1}{q-1}\sum_{l=0}^{q-2}G_{-l+\frac{q-1}{2}}G_{-2l}G_{3l}G_{\frac{q-1}{2}}T^{l-\frac{q-1}{2}}(d)T^{-3l}(c)T^{\frac{q-1}{2}}(-1)\notag\\
&=\frac{\phi(-d)}{q-1}\sum_{l=0}^{q-2}G_{-l+\frac{q-1}{2}}G_{-2l}G_{3l}G_{\frac{q-1}{2}}T^{l}(d)T^{-3l}(c).
\end{align}
Using Davenport-Hasse relation (Theorem \ref{lemma3}) for $m=2$, $\psi=T^{-l}$, we deduce that
\begin{align}\label{eq12}
G_{-l+\frac{q-1}{2}}=\frac{G_{\frac{q-1}{2}}G_{-2l}T^{l}(4)}{G_{-l}}.
\end{align}
Substituting \eqref{eq12} into \eqref{eq25} and using lemma \ref{fusi3} we deduce that
\begin{align}
D_{\frac{q-1}{2}}=\frac{q\phi(d)}{q-1}\sum_{l=0}^{q-2}\frac{G_{-2l}G_{-2l}G_{3l}}{G_{-l}}T^l\left(\frac{4d}{c^3}\right).\notag
\end{align}
Putting the values of $A, B, C$ and $D$ in \eqref{eq10} we obtain
\begin{align}
q\cdot (\#E_{1}(\mathbb{F}_q)-1)=q^2+q\phi(d)+D_{\frac{q-1}{2}},\notag
\end{align}
which yields
\begin{align}\label{eq13}
a_{q}(E_1)=-\phi(d)-\frac{\phi(d)}{q-1}\sum_{l=0}^{q-2}\frac{G_{-2l}G_{-2l}G_{3l}}{G_{-l}}T^l\left(\frac{4d}{c^3}\right).
\end{align}
Now we take $T$ to be the inverse of the Teichm\"{u}ller character, i.e., $T=\overline{\omega}$ 
and use the Gross-Koblitz formula (Theorem \ref{thm4}) to convert the above expression to an expressing involving the $p$-adic gamma function. This gives
\begin{align}
a_q(E_1)&=-\phi(d)-\frac{\phi(d)}{q-1}\sum_{l=0}^{q-2}(-p)^{\sum_{i=0}^{r-1}\{\langle\frac{-2lp^i}{q-1}\rangle+\langle\frac{-2lp^i}{q-1}\rangle+
\langle\frac{3lp^i}{q-1} \rangle -\langle\frac{-lp^i}{q-1}\rangle\}}\notag\\
&\hspace{.5cm}\times\prod_{i=0}^{r-1}\frac{\Gamma_p\left(\langle\frac{-2lp^i}{q-1}\rangle \right) 
\Gamma_p\left(\langle\frac{-2lp^i}{q-1}\rangle \right) 
\Gamma_p\left(\langle\frac{3lp^i}{q-1}\rangle \right)}{\Gamma_p\left(\langle\frac{-lp^i}{q-1}\rangle\right)}\overline{\omega}^l\left(\frac{4d}{c^3}\right).\notag
\end{align}
If we put $s=\displaystyle\sum_{i=0}^{r-1}\left\{\langle\frac{-2lp^i}{q-1}\rangle+\langle\frac{-2lp^i}{q-1}\rangle+
\langle\frac{3lp^i}{q-1}\rangle -\langle\frac{-lp^i}{q-1}\rangle\right\}$, then the above equation becomes
\begin{align}\label{eq14}
a_q(E_1)&=-\phi(d)-\frac{\phi(d)}{q-1}\sum_{l=0}^{q-2}(-p)^s\overline{\omega}^l\left(\frac{4d}{c^3}\right)\notag\\
&\hspace{.5cm}\times\prod_{i=0}^{r-1}\frac{\Gamma_p\left(\langle\frac{-2lp^i}{q-1}\rangle \right) 
\Gamma_p\left(\langle\frac{-2lp^i}{q-1}\rangle \right)\Gamma_p\left(\langle\frac{3lp^i}{q-1}\rangle \right)}{\Gamma_p\left(\langle\frac{-lp^i}{q-1}\rangle \right)}.
\end{align}
Next we use lemma \ref{lemma4} and simplify \eqref{eq14} to obtain
\begin{align}\label{eq15}
a_q(E_1)&=-\phi(d)-\frac{\phi(d)}{q-1}\sum_{l=0}^{q-2}(-p)^s\overline{\omega}^l\left(\frac{27d}{4c^3}\right)\notag\\
&\times\prod_{i=0}^{r-1}\Gamma_p\left(\left\langle\left(1-\frac{l}{q-1}\right)p^i\right\rangle\right)
\Gamma_p\left(\left\langle\left(\frac{l}{q-1}\right) p^i\right\rangle\right)\notag\\
&\times\frac{\Gamma_p\left(\left\langle\left( \frac{1}{2}-\frac{l}{q-1}\right) p^i\right\rangle\right)
\Gamma_p\left(\left\langle\left(\frac{1}{2}-\frac{l}{q-1}\right)p^i \right\rangle\right)}{\Gamma_p\left(\langle\frac{p^i}{2}\rangle\right)
\Gamma_p\left(\langle\frac{p^i}{2}\rangle \right)}\notag\\
&\times\frac{\Gamma_p\left(\left\langle\left(\frac{1}{3}+\frac{l}{q-1}\right)p^i\right\rangle\right)
\Gamma_p\left(\left\langle\left(\frac{2}{3}+\frac{l}{q-1}\right)p^i\right\rangle\right)}{\Gamma_p\left(\langle\frac{p^i}{3}\rangle\right)
\Gamma_p\left(\langle\frac{2p^i}{3}\rangle\right)}.
\end{align}
Calculating $s$ we deduce that 
\begin{align}\label{eq16}
s=\sum_{i=0}^{r-1}\left\{\lfloor\frac{-lp^i}{q-1}\rfloor -\lfloor\frac{-2lp^i}{q-1} \rfloor-
\lfloor\frac{-2lp^i}{q-1}\rfloor-\lfloor\frac{3lp^i}{q-1} \rfloor \right\}.
\end{align}
By \eqref{eq26} we have that, for $0 < l\leq q-2$,
\begin{align}
\prod_{i=0}^{r-1}\Gamma_p\left(\left\langle\left( 1-\frac{l}{q-1}\right) p^i\right\rangle\right)
\Gamma_p\left(\left\langle\left(\frac{l}{q-1} \right) p^i\right\rangle\right)=(-1)^r\overline{\omega}^l(-1).
\end{align}
Therefore,
\begin{align}
a_q(E_1)&=-\frac{q\phi(d)}{q-1}-\frac{q\phi(d)}{q-1}\sum_{l=1}^{q-2}(-p)^{s-r}\overline{\omega}^l\left(-\frac{27d}{4c^3}\right)\notag\\
&\times \prod_{i=0}^{r-1}\frac{\Gamma_p\left(\left\langle\left(\frac{1}{2}-\frac{l}{q-1}\right) p^i\right\rangle\right)
\Gamma_p\left(\left\langle\left(\frac{1}{2}-\frac{l}{q-1}\right) p^i \right\rangle\right)}{\Gamma_p\left(\langle\frac{p^i}{2}\rangle\right)
\Gamma_p\left(\langle\frac{p^i}{2}\rangle\right)}\notag\\
&\times\frac{\Gamma_p\left(\left\langle\left(\frac{1}{3}+\frac{l}{q-1}\right)p^i\right\rangle\right)
\Gamma_p\left(\left\langle\left(\frac{2}{3}+\frac{l}{q-1}\right)p^i\right\rangle\right)}{\Gamma_p\left(\langle\frac{p^i}{3}\rangle\right)
\Gamma_p\left(\langle\frac{2p^i}{3}\rangle\right)}.\notag
\end{align}
Now using the following relation for $0\leq l \leq q-2$
\begin{align}
&\prod_{i=0}^{r-1}\Gamma_p\left(\left\langle\left(\frac{1}{3}+\frac{l}{q-1}\right)p^i\right\rangle \right)
\Gamma_p\left(\left\langle\left(\frac{2}{3}+\frac{l}{q-1}\right)p^i\right\rangle \right)\notag\\
&=\prod_{i=0}^{r-1}\Gamma_p\left(\left\langle\left(-\frac{1}{3}+\frac{l}{q-1}\right)p^i\right\rangle \right)
\Gamma_p\left(\left\langle\left(-\frac{2}{3}+\frac{l}{q-1}\right)p^i\right\rangle \right)\notag\\
\end{align}
and lemma \ref{lemma5}, we deduce that
\begin{align}
a_q(E_1)&=-\frac{q\phi(d)}{q-1}\sum_{l=0}^{q-2}\prod_{i=0}^{r-1}(-p)^{-\lfloor \langle \frac{p^i}{2} \rangle -\frac{lp^i}{q-1} \rfloor
-\lfloor \langle\frac{p^i}{2} \rangle -\frac{lp^i}{q-1} \rfloor -\lfloor \langle -\frac{p^i}{3} \rangle + \frac{lp^i}{q-1} \rfloor
 -\lfloor \langle -\frac{2p^i}{3} \rangle + \frac{lp^i}{q-1} \rfloor}   \notag\\
 &\times\frac{\Gamma_p\left(\left\langle\left( \frac{1}{2}-\frac{l}{q-1}  \right) p^i\right\rangle\right)
\Gamma_p\left(\left\langle\left(\frac{1}{2}-\frac{l}{q-1}   \right) p^i \right\rangle\right)  }{\Gamma_p\left(\langle\frac{p^i}{2}\rangle  \right)
\Gamma_p\left(\langle\frac{p^i}{2}\rangle\right)}\notag\\
&\times\frac{\Gamma_p\left(\left\langle\left(-\frac{1}{3}+\frac{l}{q-1}\right)p^i\right\rangle\right)
\Gamma_p\left(\left\langle\left(-\frac{2}{3}+\frac{l}{q-1}\right)p^i\right\rangle\right)}{\Gamma_p\left(\langle-\frac{p^i}{3}\rangle\right)
\Gamma_p\left(\langle-\frac{2p^i}{3}\rangle\right)}\overline{\omega}^l\left(-\frac{27d}{4c^3}  \right)\notag\\
&=q\cdot \phi(d)\cdot {_2}G_2\left[ \begin{array}{cc}
                          \frac{1}{2}, & \frac{1}{2} \\
                          \frac{1}{3}, & \frac{2}{3}
                        \end{array}|-\frac{27d}{4c^3}
 \right]_q.\notag
\end{align}
This completes the proof of the theorem.
\end{proof}
\textbf{Proof of Theorem \ref{thm1}:} We have $j(E_{a, b})\neq 0$. Hence $a\neq 0$. Since $(-a/3)$ is a quadratic residue in $\mathbb{F}_q$, 
we find $k\in \mathbb{F}_{q}^{\times}$ such that
$3k^2+a=0$. A change of variables $(x,y)\mapsto (x+k,h)$ takes the elliptic curve $E_{a, b}:y^2=x^3+ax+b$ to
\begin{align}\label{eq23}
E^{\prime}:y^2=x^3+3kx^2+(k^3+ak+b).
\end{align}
Clearly $a_{q}(E_{a,b})=a_{q}(E^{\prime})$. Using Theorem \ref{thm3} for the elliptic curve $E'$, we complete the proof.\\
Now, we are going to prove Theorem \ref{thm2}. The proof will follow as a consequence of the next theorem. 
We consider an elliptic curve $E_2$ over $\mathbb{F}_q$ in the form
\begin{align}
E_2 : y^2=x^3+fx^2+gx\notag,
\end{align}
where $f\neq0$. We express the trace of Frobenius endomorphism on the curve $E_2$ as a special value of the function ${_2}G_2[\cdots]$ in the following way.
\begin{thm}\label{thm5}
Let $q=p^r$, $p>3$ be a prime. The trace of Frobenius on $E_2$ is given by
\begin{align}
a_q(E_2)=q\cdot\phi(-g)\cdot {_2}G_2\left[ \begin{array}{cc}
                          \frac{1}{2}, & \frac{1}{2} \\
                          \frac{1}{4}, & \frac{3}{4}
                        \end{array}|\frac{4g}{f^2}
 \right]_q.\notag
\end{align}
\end{thm}
\begin{proof}
We recall that $\#E_2(\mathbb{F}_q)-1=\#\{(x,y)\in\mathbb{F}_q\times\mathbb{F}_q : y^2=x^3+fx^2+gx\}.$\\
Let $P(x,y)=x^3+fx^2+gx-y^2$. Now using the identity
\begin{align}
\sum_{z\in\mathbb{F}_q}\theta(zP(x,y))=\left\{
                                         \begin{array}{ll}
                                           q, & \hbox{if $P(x,y)=0$;} \\
                                           0, & \hbox{if $P(x,y)\neq0$,}
                                         \end{array}
                                       \right.
\end{align}
we obtain
\begin{align}\label{eq17}
q\cdot (\#E_2(\mathbb{F}_q)-1)&=\sum_{x,y,z\in\mathbb{F}_q}\theta(zP(x,y))\notag\\
&=q^2+\sum_{z\in\mathbb{F}_{q}^{\times}}\theta(0)+\sum_{y,z\in\mathbb{F}_{q}^{\times}}\theta(-zy^2)+\sum_{x,z\in\mathbb{F}_{q}^{\times}}\theta(zx^3)\theta(zfx^2)
\theta(zgx)\notag\\
&+\sum_{x,y,z\in\mathbb{F}_{q}^{\times}}\theta(zx^3)\theta(zfx^2)\theta(zgx)\theta(-zy^2)\notag\\
&=q^2+(q-1)+A+B+C.
\end{align}
From \cite[Theorem 3.2]{BK2} we have $A=-(q-1)$ and $C=-B+C_{\frac{q-1}{2}}$, where
\begin{align}
C_{\frac{q-1}{2}}&=\frac{G_{\frac{q-1}{2}}}{(q-1)^3}\sum_{l,m,n=0}^{q-2}G_{-l}G_{-m}G_{-n}T^m(f)T^n(g)T^{\frac{q-1}{2}}(-1)\notag\\
&\times \sum_{x\in\mathbb{F}_{q}^{\times}}T^{3l+2m+n}(x)\sum_{z\in\mathbb{F}_{q}^{\times}}T^{l+m+n+\frac{q-1}{2}}(z),\notag
\end{align}
which is non zero only if $n=l$ and $m=-2l+\frac{q-1}{2}$. This gives
\begin{align}
C_{\frac{q-1}{2}}=\frac{G_{\frac{q-1}{2}}\phi(-1)}{q-1}\sum_{l=0}^{q-2}G_{-l}G_{2l+\frac{q-1}{2}}G_{-l}T^l\left(\frac{g}{f^2}\right).\notag
\end{align}
Substituting the values of $A$, $B$ and $C$ in \eqref{eq17} we obtain
\begin{align}\label{eq19}
q\cdot (\#E_2(\mathbb{F}_q)-1)=q^2+\frac{G_{\frac{q-1}{2}}\phi(-1)}{q-1}\sum_{l=0}^{q-2}G_{-l}G_{2l+\frac{q-1}{2}}G_{-l}T^l\left(\frac{g}{f^2}\right).
\end{align}
Replacing $l$ by $l-\frac{q-1}{2}$ we deduce that
\begin{align}\label{eq27}
q\cdot (\#E_2(\mathbb{F}_q)-1)&=q^2+\frac{G_{\frac{q-1}{2}}\phi(-1)}{q-1}\sum_{l=0}^{q-2}G_{-l+\frac{q-1}{2}}G_{2l+\frac{q-1}{2}}G_{-l+\frac{q-1}{2}}
T^{l-\frac{q-1}{2}}\left(\frac{g}{f^2}\right)\notag\\
&=q^2+\frac{G_{\frac{q-1}{2}}\phi(-g)}{q-1}\sum_{l=0}^{q-2}G_{-l+\frac{q-1}{2}}G_{2l+\frac{q-1}{2}}G_{-l+\frac{q-1}{2}}
T^{l}\left(\frac{g}{f^2}\right).
\end{align}
Using Davenport-Hasse relation \ref{lemma3} for $m=2$, $\psi=T^{-l}$ and $\psi=T^{2l}$ successively, we have
\begin{align}
&G_{-l+\frac{q-1}{2}}=\frac{G_{\frac{q-1}{2}}G_{-2l}T^l(4)}{G_{-l}}\notag
\end{align}
and
\begin{align}
&G_{2l+\frac{q-1}{2}}=\frac{G_{4l}G_{\frac{q-1}{2}}T^{-l}(16)}{G_{2l}}.\notag
\end{align}
Putting these values in \eqref{eq27} and using lemma \ref{fusi3}, we obtain
\begin{align}
q\cdot (\#E_2(\mathbb{F}_q)-1)=q^2+\frac{q^2\phi(-g)}{q-1}\sum_{l=0}^{q-2}\frac{G_{-2l}G_{-2l}G_{4l}}{G_{-l}G_{-l}G_{2l}}T^l\left(\frac{g}{f^2}\right).\notag
\end{align}
We now put $T=\overline{\omega}$. Then \eqref{eq18} and Gross-Koblitz formula (Theorem \ref{thm4}) yield
\begin{align}\label{eq21}
a_q(E_2)&=-\frac{q\phi(-g)}{q-1}\sum_{l=0}^{q-2}(-p)^{\sum_{i=0}^{r-1}\{2\langle-\frac{2lp^i}{q-1}\rangle+\langle\frac{4lp^i}{q-1}\rangle
-2\langle-\frac{lp^i}{q-1} \rangle -\langle\frac{2lp^i}{q-1}\rangle\}}\notag\\
&\times \prod_{i=0}^{r-1}\frac{\Gamma_p\left(\langle\frac{-2lp^i}{q-1}\rangle \right) 
\Gamma_p\left(\langle\frac{-2lp^i}{q-1}\rangle \right) 
\Gamma_p\left(\langle\frac{4lp^i}{q-1}\rangle \right)}{\Gamma_p\left(\langle\frac{-lp^i}{q-1}\rangle \right)
\Gamma_p\left(\langle\frac{-lp^i}{q-1}\rangle \right) 
\Gamma_p\left(\langle\frac{2lp^i}{q-1}\rangle \right)}\overline{\omega}^l\left(\frac{g}{f^2}\right)\notag\\
&=-\frac{q\phi(-g)}{q-1}\sum_{l=0}^{q-2}(-p)^s\overline{\omega}^l\left(\frac{g}{f^2}\right)\notag\\
&\times\prod_{i=0}^{r-1}\frac{\Gamma_p\left(\langle\frac{-2lp^i}{q-1}\rangle \right) 
\Gamma_p\left(\langle\frac{-2lp^i}{q-1}\rangle \right) 
\Gamma_p\left(\langle\frac{4lp^i}{q-1} \rangle \right)}{\Gamma_p\left(\langle\frac{-lp^i}{q-1}\rangle \right)
\Gamma_p\left(\langle\frac{-lp^i}{q-1}\rangle \right) \Gamma_p\left(\langle\frac{2lp^i}{q-1}\rangle \right)},
\end{align}
where $s=\displaystyle\sum_{i=0}^{r-1}\left\{2\langle\frac{-2lp^i}{q-1}\rangle+\langle\frac{4lp^i}{q-1}\rangle-
2\langle\frac{-lp^i}{q-1}\rangle -\langle\frac{2lp^i}{q-1}\rangle\right\}$.
Next we use lemma \ref{lemma4} and after simplification we obtain
\begin{align}\label{eq22}
a_q(E_2)&=-\frac{q\phi(-g)}{q-1}\sum_{l=0}^{q-2}(-p)^s\overline{\omega}^l\left(\frac{4g}{f^2}\right)\notag\\
&\times \prod_{i=0}^{r-1}\frac{\Gamma_p\left(\left\langle\left(\frac{1}{2}- \frac{l}{q-1}  \right)p^i  \right\rangle \right)
\Gamma_p\left(\left\langle\left(\frac{1}{2}- \frac{l}{q-1}  \right)p^i  \right\rangle \right)}{\Gamma_p\left(\langle \frac{p^i}{2}\rangle \right)
\Gamma_p\left(\langle \frac{p^i}{2}\rangle \right)}\notag\\
&\frac{\Gamma_p\left(\left\langle\left(\frac{1}{4} +\frac{l}{q-1}\right)p^i\right\rangle \right)
\Gamma_p\left(\left\langle\left(\frac{3}{4} +\frac{l}{q-1}  \right)p^i  \right\rangle \right)}
{\Gamma_p\left(\langle \frac{p^i}{4}\rangle \right)\Gamma_p\left(\langle \frac{3p^i}{4}\rangle \right)}.
\end{align}
We now simplify the expression for $s$ and find that
\begin{align}\label{eq22}
s=\sum_{i=0}^{r-1}\left\{\lfloor\frac{2lp^i}{q-1}\rfloor +2\lfloor\frac{-lp^i}{q-1} \rfloor
-2\lfloor\frac{-2lp^i}{q-1} \rfloor-\lfloor\frac{4lp^i}{q-1}\rfloor \right\}.
\end{align}
The following relation for $0\leq l\leq q-2$
\begin{align}
&\prod_{i=0}^{r-1}\Gamma_p\left(\left\langle\left(\frac{1}{4} +\frac{l}{q-1}  \right)p^i  \right\rangle \right)\Gamma_p\left(\left\langle\left(\frac{3}{4} +\frac{l}{q-1}  \right)p^i  \right\rangle \right)\notag\\
&=\prod_{i=0}^{r-1}\Gamma_p\left(\left\langle\left(-\frac{1}{4} +\frac{l}{q-1}  \right)p^i  \right\rangle \right)
\Gamma_p\left(\left\langle\left(-\frac{3}{4} +\frac{l}{q-1}  \right)p^i  \right\rangle \right)\notag
\end{align}
and lemma \ref{lemma6} yield
\begin{align}
a_q(E_2)&=-\frac{q\phi(-g)}{q-1}\sum_{l=0}^{q-2}\prod_{i=0}^{r-1}(-p)^{-\lfloor \langle \frac{p^i}{2} \rangle -\frac{lp^i}{q-1} \rfloor
-\lfloor \langle\frac{p^i}{2} \rangle -\frac{lp^i}{q-1} \rfloor -\lfloor \langle -\frac{p^i}{4} \rangle + \frac{lp^i}{q-1} \rfloor
 -\lfloor \langle -\frac{3p^i}{4} \rangle + \frac{lp^i}{q-1} \rfloor}\notag\\
 &\times\frac{\Gamma_p\left(\left\langle\left(\frac{1}{2}- \frac{l}{q-1}  \right)p^i  \right\rangle \right)
\Gamma_p\left(\left\langle\left(\frac{1}{2}- \frac{l}{q-1}  \right)p^i  \right\rangle \right)}{\Gamma_p\left(\langle \frac{p^i}{2}\rangle \right)
\Gamma_p\left(\langle \frac{p^i}{2}\rangle \right)}\notag\\
&\times\frac{\Gamma_p\left(\left\langle\left(-\frac{1}{4} +\frac{l}{q-1}  \right)p^i  \right\rangle \right)
\Gamma_p\left(\left\langle\left(-\frac{3}{4} +\frac{l}{q-1}  \right)p^i  \right\rangle \right)}
{\Gamma_p\left(\langle -\frac{p^i}{4}\rangle \right)\Gamma_p\left(\langle -\frac{3p^i}{4}\rangle \right)}\overline{\omega}^l\left(\frac{4g}{f^2}\right)\notag\\
&=q\cdot \phi(-g)\cdot {_2}G_2\left[ \begin{array}{cc}
                          \frac{1}{2}, & \frac{1}{2} \\
                          \frac{1}{4}, & \frac{3}{4}
                        \end{array}|\frac{4g}{f^2}
 \right]_q.\notag
\end{align}
This completes the proof of the theorem.
\end{proof}
\textbf{Proof of Theorem \ref{thm2}:} Here $j(E_{a,b})\neq 1728$ and hence $b\neq 0$. 
Let $h\in\mathbb{F}_{q}^{\times}$ be such that $h^3+ah+b=0$. A change of variables $(x,y)\mapsto(x+h,y)$ takes the elliptic curve $E_{a,b}:y^2=x^3+ax+b$ to
\begin{align}\label{eq24}
E^{\prime\prime}: y^2=x^3+3hx^2+(3h^2+a)x.
\end{align}
Clearly $a_q(E_{a, b})=a_q(E^{\prime\prime})$ and $3h\neq0$. Using Theorem \ref{thm5} for the elliptic curve $E''$, we complete the proof.

\section*{Acknowledgment}
We thank Dipendra Prasad and Ken Ono for careful reading of a draft of the manuscript.


\begin{thebibliography}{99}
\bibitem{BK1}
R. Barman and G. Kalita, {\it Elliptic curves and special values of Gaussian hypergeometric series},
J. Number Theory 133 (2013), 3099--3111.

\bibitem{BK2}
R. Barman and G. Kalita, {\it Hypergeometric functions over $\mathbb{F}_q$ and traces of Frobenius for elliptic curves},
Proc. Amer. Math. Soc. 141 (2013), 3403--3410.

\bibitem{evans}
B. Berndt, R. Evans, and K. Williams, {\it Gauss and Jacobi Sums}, Canadian Mathematical Society Series of Monographs and Advanced Texts, 
A Wiley-Interscience Publication, John Wiley \& Sons, Inc., New York, 1998.

\bibitem {Fuselier} J. Fuselier, \textit{Hypergeometric functions over $\mathbb{F}_p$ and relations to elliptic curve and modular forms},
Proc. Amer. Math. Soc. 138 (2010), 109--123.

\bibitem{greene}
J. Greene, {\it Hypergeometric functions over finite fields}, Trans. Amer. Math. Soc. 301 (1) (1987), 77--101.

\bibitem{gross}
B. H. Gross and N. Koblitz, {\it Gauss sum and the $p$-adic $\Gamma$-function}, Annals of Mathematics 109 (1079), 569-581.  

\bibitem{ireland}
K. Ireland and M. Rosen, {\it A Classical Inroduction to Modern Number Theory}, Springer International Edition, Springer, 2005.

\bibitem {Lang} S. Lang, \textit{Cyclotomic Fields I and II},
Graduate Texts in Mathematics, vol. 121, Springer-Verlag, New York, 1990.

\bibitem {lennon} C. Lennon, \textit{Gaussian hypergeometric evaluations of traces of Frobenius for elliptic curves},
Proc. Amer. Math. Soc. 139 (2011), 1931--1938.

\bibitem {lennon2} C. Lennon, \textit{Trace formulas for Hecke operators, Gaussian hypergeometric functions, and the modularity of a threefold},
J. Number Theory 131 (12) (2011), 2320--2351.

\bibitem{mccarthy2}
D. McCarthy, {\it The trace of Frobenius of elliptic curves and the p-adic gamma function}, Pacific J. Math. 261 (1) (2013), 219--236.

\end{thebibliography}
\end{document}